\documentclass[a4paper,11pt]{article}
\usepackage{enumerate,color}
\usepackage{graphicx}
\usepackage{hyperref}
\usepackage{xcolor}
\usepackage{amsmath,amssymb,euscript,graphicx,amsfonts,verbatim}

\newtheorem{theorem}{Theorem}[section]

\newtheorem{proposition}[theorem]{Proposition}

\newtheorem{corollary}[theorem]{Corollary}

\newtheorem{example}[theorem]{Example}
\newtheorem{remark}[theorem]{Remark}

\newcommand{\Qed}{\rule{2.5mm}{3mm}}
\newenvironment{proof}{{\noindent \sc Proof.}}{\hfill $\Qed$ \\}

\newcommand{\ZZ}{\mathbb{Z}}

\newcommand{\hsh}{^{\#}}

\newcounter{case}

\renewcommand{\thecase}{\arabic{case}}

\newcounter{subcase}

\numberwithin{subcase}{case}

  \def\G{\Gamma}

\begin{document}


\begin{center}
{\bf\large
ON 3-ISOREGULARITY OF MULTICIRCULANTS}
\end{center}


	\begin{center}	
		Klavdija Kutnar{\small$^{a,b,}$}\footnotemark  \ 
		Dragan Maru\v si\v c{\small$^{a, b, c,}$}\footnotemark \
		 and 
		\addtocounter{footnote}{0} 
		\v Stefko Miklavi\v c{\small$^{a,b,c,*,}$}\footnotemark
	\\ [+2ex]
	{\it \small 
		$^a$University of Primorska, UP IAM, Muzejski trg 2, 6000 Koper, Slovenia\\
		$^b$University of Primorska, UP FAMNIT, Glagolja\v ska 8, 6000 Koper, Slovenia\\
		$^c$IMFM, Jadranska 19, 1000 Ljubljana, Slovenia}
\end{center}

\addtocounter{footnote}{-2}
\footnotetext{The work of Klavdija Kutnar  is supported in part by the Slovenian Research and Innovation Agency (research program P1-0285 and research projects J1-3001, N1-0209, N1-0391, J1-50000, J1-60012).}
\addtocounter{footnote}{1}
\footnotetext{
	The work of Dragan Maru\v si\v c is supported in part by the Slovenian Research and Innovation Agency (I0-0035, research program P1-0285 and research projects J1-3001, J1-50000).}
\addtocounter{footnote}{1}
\footnotetext{The work of \v Stefko Miklavi\v c is supported in part by the Slovenian Research and Innovation Agency (research program P1-0285 and research projects J1-3001, J1-3003, J1-4008, J1-4084, N1-0208, N1-0353, J1-50000, J1-60012).

	~*Corresponding author e-mail:~stefko.miklavic@upr.si}



\begin{abstract}
A graph is said to be $k$-{\em isoregular} if any two vertex subsets 
of cardinality at most $k$, that induce subgraphs of the same isomorphism type, have the same number of neighbors.  
It is shown that no $3$-isoregular bicirculant 
(and more generally, no locally $3$-isoregular bicirculant)
of order twice an odd number exists. 
Further, partial results for bicirculants of order twice an even number  
as well as tricirculants  of specific orders, are also obtained.
Since $3$-isoregular graphs are necessarily strongly regular, the
above result about bicirculants, among other, brings us a step closer to obtaining a direct 
proof of a classical consequence of the Classification of Finite Simple Groups that no simply primitive  group of degree twice  a prime exists for primes greater than $5$. 
\end{abstract}




\section{Introductory remarks}
\label{sec:intro}
\indent

As a meeting point of combinatorics, geometry and algebra, strongly regular graphs have been in the interest of mathematical community for quite a long time, with first papers dating several decades ago \cite{Bo63}. 
Recall, that a {\em strongly regular graph} $\G$ {\em with parameters }
$(n,k,\lambda,\mu)$, alternatively an 
$(n,k,\lambda,\mu)$-{\em strongly regular graph},
is a regular graph of order $n$ and valency  $k$ such
that the number $\lambda$ of $3$-cliques $K_3$ in $\G$ containing a given edge is independent of the choice of edge and the number $\mu$ of 
$2$-claws $K_{1,2}$ in $\G$ containing a given non-edge is independent of the choice of non-edge. Note that the complement 
$\overline{\G}$ of $\G$ is also a strongly regular graph, with parameters

\begin{equation}
\label{eq:srg-complement}
(n,l,\lambda',\mu') = (n,n-k-1,n-2-2k+\mu, n-2k+\lambda).
\end{equation} 

\noindent
A strongly regular graph $\G$ is said to be {\em non-trivial} 
if $n \ge 2$ and both $\G$ and its complement $\overline{\G}$ are connected graphs.
Hereafter, we restrict ourselves to non-trivial strongly regular graphs.
Also, the following equality for an 
$(n,k,\lambda,\mu)$- strongly regular graph is well known:
\begin{equation}
\label{eq:srg}
k(k-\lambda-1)=\mu(n-1-k).
\end{equation} 

Strongly regular graphs exist in such an abundance that a complete classification seems hopeless
(see for example \cite{F02,M07,W71}).
It is therefore sensible to restrict oneself  to certain more tractable  subclasses of strongly regular graphs. This explains an increased interest in strongly regular graphs satisfying certain additional regularity conditions. 
In this respect different directions have been taken. 
One approach is using the concept of a $t$-vertex condition. 
Following \cite{HH70, RS15} we say that a graph $\G$ satisfies the 
$t$-{\em vertex condition} for some $t > 1$, if for every positive integer $j \le t$ the number of $j$-vertex subgraphs of $\G$ of a given (isomorphism) type containing a given pair 
$(u,v)$ of vertices depends solely on whether $u$ and $v$ are equal, adjacent or non-adjacent. Clearly, the $2$-vertex condition is satisfied if and only if the graph is regular, and the $3$-vertex condition is satisfied if and only if the graph is strongly regular. Since all rank $3$ graphs 
(that is, orbital graphs of transitive permutation groups of rank $3$)
satisfy the $t$-vertex condition for every  $t > 1$, an ongoing research effort has been directed towards getting  constructions of non-rank $3$ strongly regular graphs with the $t$-vertex condition  for $t\geq 4$ 
(see \cite{BIK23,I89,I94,KMRR05,PC21,RS00,RS15}).
It was conjectured  that there is a number $t_0$ such that no non-rank $3$ graph satisfies the $t_0$-vertex condition (see \cite{FKM94}).
Since  the point graphs of generalized quadrangles $GQ(s,s^2)$, $s\ge 2$, 
are proved to satisfy the $7$-vertex condition \cite[Theorem~2]{RS15}, it follows that $t_0 \geq 8$.

Another important restriction imposed on strongly regular graphs is via the concept of isoregularity. For a graph $\G$ let $S$ be a subset of the vertex set $V(\G)$. A vertex $v \in V(\G)$ is said to be a 
{\em neighbor} of $S$ if it is a neighbor of every vertex of $S$, and the number of neighbors of $S$ is the {\em valency} of $S$. Given a positive integer $0 < k \leq  |V(\G)|$, if for any $j$-element subset $S$, $j \leq k$, the valency of $S$ depends solely on the isomorphism type of the subgraph induced by $S$, we say that $\G$ is $k$-{\em isoregular}. Clearly, a graph is $1$-isoregular if and only if its regular, and it is $2$-isoregular if and only if it is strongly regular. The isoregularity condition is quite restrictive for it is known that every $5$-isoregular finite graph is {\em homogeneous} (see \cite{PC80,AG76,GK78}), 
that is, in such a graph every isomorphism between two subgraphs extends 
to an automorphism of the graph. This puts the main focus of interest on 
$3$-isoregular and $4$-isoregular graphs which are not homogeneous. 
While Mc Laughlin graph on $275$ vertices (see \cite{PC21}) 
is the only known $4$-isoregular graph, there are several constructions of 
$3$-isoregular graphs, which are not homogeneous.
For example,  the above mentioned point graphs of generalized quadrangles $GQ(s,s^2)$, as well as  certain triangle-free point graphs of partial quadrangles are $3$-isoregular (see \cite{PC21,RS15} for more details).
Note that generalized quadrangles are a subclass of partial quadrangles, 
where in the latter we are requiring that a point not on a line has at most one collinear point on this line as opposed to the generalized quadrangles where there is exactly one such point.

This brings us to the main purpose of this paper, where we 
approach the $3$-isoregularity restriction on strongly regular graphs via  
semiregular automorphisms.  Recall that a nonidentity automorphism  of a graph on $mn$ vertices is $(m,n)$-{\em semiregular} if it has $m$ cycles of length $n$ in its cycle decomposition. 
It is worth mentioning  a long standing conjecture 
regarding vertex-transitive graphs admitting semiregular automorphisms
(see \cite{BaGS25, GV20, KM22, DM81} for more details and updates).
A graph admitting a  semiregular automorphism
is referred to as a {\em multicirculant} (and  $n$-{\em multicirculant}
when the size of the corresponding cycles in the cycle decompositions, hereafter called {\em orbits}, needs to be specified).
In particular, for $m=1,2,3$ and $4$, 
the corresponding multicirculants are called, respectively, 
{\em n-circulants}, {\em n-bicirculants}, {\em n-tricirculants} 
and {\em n-tetracirculants}. While everything is known about strong regularity of circulants (see \cite{BM79}), the situation with  bicirculants 
and tricirculants is still quite open, save for some structural results 
and some constructions (see \cite{deRJ92,KMMS09,LM93,M84}). 
It is well known that a strongly regular circulant is necessarily a Paley graph, and it is easily checked that, with the exception of the $5$-cycle, 
Paley graphs are not $3$-isoregular. In this paper we show 
that there are no 
$3$-isoregular graphs among $n$-bicirculants, $n$ odd,  
and that there are no $3$-isoregular graphs among 
$n$-tricirculants, $n$ prime or coprime 
to a particular number depending on the 
corresponding strongly regular parameters 
(see Corollaries~\ref{cor:3iso-2} and~\ref{cor:3iso-3}). 
In fact we prove a somewhat more general result, that is,
we prove
non-existence of the so-called locally $3$-isoregular
bicirculants and tricirculants with above order restrictions
(see Theorems~\ref{thm:bicirc3}
and~\ref{thm:tri3}, and Section~\ref{sec:locally} 
for the definition of locally $3$-isoregular graphs).
Let us also mention that, apart from the fact that our results hold
for local $3$-isoregularity, the proofs do not depend on
the $3$-isoregularity results from \cite[Theorem~6.5.]{CGS78}.
As for $3$-isoregular $n$-bicirculants, $n$ even,
they do not exist unless $n$ is divisible by $8$.
Furthermore, there are two such graphs of order $16$, which happens to be 
the smallest admissible order for $3$-isoregular graphs. One of them is 
the Clebsch graph, which is 
the folded $5$-cube graph and also
the point graph of the partial quadrangle  
$PQ(4,1,2)$, with  parameters $(16,5,0,2)$. 
The second graph  is the Cartesian product 
$K_4 \square K_4$, also known as
the line graph of the complete bipartite graph
$K_{4,4}$, with  parameters $(16,6,2,2)$.
By extensions these two graphs  are also examples of $3$-isoregular 
$4$-tetracirculants.
This  makes
both the class of $n$-bicirculants, $n$ even (see \cite{LM93}
for admissible parameters of these graphs), and
tetracirculants as the next natural target in the investigation 
of $3$-isoregular graphs.
 

As a final remark, since $3$-isoregular graphs 
are a subclass of strongly regular graphs that satisfy the $4$-vertex condition
 -- and the latter is always satisfied for rank $3$ graphs --
 the above  mentioned result about non-existence of  locally
$3$-isoregular  
 $n$-bicirculants, $n$ odd, brings us a step closer to obtaining a direct 
proof of a classical consequence of the Classification of Finite Simple Groups 
(CFSG) that no simply primitive  group of degree twice  a prime exists for primes greater than $5$ (see \cite{PC81,LS85,N09,HW56}).  
This would be achieved if one manages to show that no
$n$-bicirculant, $n > 5$ a prime, satisfies the $4$-vertex condition.
In fact, we believe that this statement holds for all $n > 5$ odd.

%
%


\section{Strongly regular multicirculants}
\label{sec:multi}
\indent

As outlined in the introductory section, 
in this paper our main interest is in strongly regular and, 
more specifically, $3$-isoregular multicirculants 
with semiregular automorphisms having  a small number of orbits.
Strongly regular circulants have been classified,
they are necessarily Paley graphs (see Proposition~\ref{pro:sr-circ}). 
As for strongly regular bicirculants and tricirculants, the situation is quite a bit more complicated, as explained below.

\subsection{Circulants}
\label{sub:cir}
\indent

Using the complete classification of strongly regular circulants, 
which was independently achieved by Bridges and Mena 
\cite{BM79} (extensively using the results of Kelly \cite{K54}), 
Hughes, van Lint and Wilson \cite{HLW79}, Ma \cite{M84}, and partially 
by Maru\v si\v c \cite{M89}, 
and applying the characterization of $3$-isoregular graphs in terms of  the induced subgraphs on the neighbors' and non-neighbors' sets 
(see Proposition~\ref{pro:3iso-char})
one can easily see that the $5$-cycle is the
only nontrivial $3$-isoregular strongly regular circulant 
(see Corollary~\ref{cor:3iso-circ}).

\begin{proposition}
\label{pro:sr-circ}
{\rm \cite{BM79,HLW79,M84}}
If $\Gamma$ is a nontrivial strongly regular circulant,
then $\Gamma$ is isomorphic to a Paley graph $P(p)$
for some prime $p\equiv {1\pmod 4}$. 
\end{proposition}

For a graph $\G$, a vertex $v$ and an integer $i$, the 
{\em $i$-th subconstituent} $\G_i(v)$ is defined as the subgraph
of $\G$ induced by all vertices at distance $i$ from $v$.

\begin{proposition}
\label{pro:3iso-char}
{\rm \cite[Proposition 4]{RS15}}
A strongly regular graph $\G$ is $3$-isoregular if and only if  
the subconstituents  $\G_1(v)$ and $\G_2(v)$ 
are strongly regular with parameters which do not depend 
on the choice of vertex $v$.
\end{proposition}

In the next result we will show that the only nontrivial $3$-isoregular strongly regular circulant graph is the five cycle $C_5$. The argument makes a direct use of the above two propositions and the well-known Hoffman's bound, which we now briefly recall for the case of strongly regular graphs (see for example \cite[Proposition 1.3.2]{BCN} or \cite[Formula (3.22)]{delsarte}). If $C$ is a clique in a strongly regular graph $\Gamma$ with valency $k$ and  smallest eigenvalue $-m$, then the Hoffman's bound states that
\begin{equation}
	\label{hoffman}
  |C| \le 1 + \frac{k}{m}.
\end{equation}

\begin{corollary}
\label{cor:3iso-circ}
Let $\Gamma=\mathrm{Circ}(n,S)$ be a nontrivial strongly regular
graph with parameters $(n,k,\lambda,\mu)$. Then $\Gamma$ is not
$3$-isoregular unless $\Gamma\cong C_5$. 
\end{corollary}

\begin{proof}
Let $\Gamma$ be a $3$-isoregular circulant. First  of all, by Proposition~\ref{pro:sr-circ}, $\Gamma$ is a Paley graph of prime order 
$p\equiv {1\pmod 4}$. It is clear that $C_5$ is 3-isoregular, so assume that $p \ge 13$. Pick a vertex $v$ of $\Gamma$. The corresponding subconstituents $\Gamma_1(v)$ and $\Gamma_2(v)$ 
are both circulants of order $(p-1)/2$, and they are strongly regular by Proposition \ref{pro:3iso-char}. Since $(p-1)/2$ is not a prime, $\Gamma_1(v)$ and $\Gamma_2(v)$  need to be trivial strongly regular graphs (a disjoint union of complete graphs $mK_t$, or a complete multipartite graph $K_{m \times t}$). Observe also that the valency of $\Gamma_1(v)$ is $(p-5)/4$. If $\Gamma_1(v) \cong K_{m \times t}$, then using $mt=(p-1)/2$ and $(m-1)t=(p-5)/4$ we get $m=(2p-2)/(p+3)=  2 -8/(p+3)$, contradicting the integrality of $m$. If $\Gamma_1(v) \cong mK_t$, then we have $t-1=(p-5)/4$, and so $\Gamma$ contains a clique of size $t+1=(p+3)/4$, contradicting \eqref{hoffman}
\end{proof}

\subsection{Bicirculants}
\label{sub:bi}
\indent

Let $n$ be a positive integer, $n \geq  2$, 
and let $\ZZ_n^*$ and $\ZZ_n\hsh$ denote, respectively,
the group of all invertible elements and the set of all nonzero elements 
in $\ZZ_n$. Further, for a subset $A$ of $\ZZ_n$ 
we let $A\hsh = A \setminus \{0\}$, $A^c = \ZZ_n \setminus A$ and
we let $\widehat{A} = \ZZ_n\hsh \setminus A = (A^c)\hsh$.

Let $\G$ be an $n$-bicirculant, and  
$\rho$ a  $(2,n)$-semiregular automorphism of $\G$ with orbits  
$U$ and $W$. We can represent $\G$ by a triple of subsets of $\ZZ_n$
in the following way. Let $u \in U$ and $w \in W$.
Let $S=\{s \in \ZZ_n \hsh \mid u \sim \rho^s(u)\}$
be the symbol of the $n$-circulant induced on $U$ (relative to $\rho$)
and similarly, and let 
$S' = \{r\in Z_n\hsh \mid w  \sim \rho^r(w)\} $
be the symbol of the $n$-circulant induced on $W$ (relative to $\rho$). Moreover, let $T =\{t \in \ZZ_n \mid u \sim \rho^t(w)\}$
define the adjacencies between the two orbits $U$ and $W$.
The ordered triple $[S,S',T]$
is called the {\em symbol} of $\G$ relative to 
the triple $(\rho, u,w)$.
Note that $S = - S$ and $S' = -S'$ are symmetric, 
that is, inverse-closed subsets of $\ZZ_n\hsh$, and are
independent of the particular choice of vertices $u$ and $w$.

The above ``triple-representation'' is particularly nice in 
case of strongly regular $n$-bicirculants, $n$ odd.
Namely, in this case the above two sets $S$ and $S'$ are complementary in 
$\ZZ_n\hsh$, a fact that proves to be  crucial for the purpose of this paper.
Using finite Fourier transform it was first proved in \cite{DM88} 
for the case when $n$ is a prime, and later extended also to the case when $n$ is odd in \cite{deRJ92}.
The result below is a transcription of  
\cite[Theorem~4.4]{deRJ92}
from their group rings terminology into  our graph-theoretic language.
We would also like to remark that the result about the parameters of 
 strongly regular graphs is explicitly stated in \cite[Theorem~4.4]{deRJ92}, 
whereas the fact that the two symbols $S$ and $S'$ are complementary 
can be deduced from the proof itself.  

\begin{theorem}
\label{the:bicirc}
Let $n$ be odd and let $\G$ be a non-trivial strongly regular $n$-bicirculant with parameters $(2n,k,\lambda,\mu)$.
Then $2n = (2m+1)^2 + 1$ for some positive integer $m$ and
(up to taking complements) we have 
$k = m(2m+1)$, $\lambda = m^2-1$ and $\mu = m^2$. 
Further, there exist $S = -S \subseteq \ZZ_n\hsh$ 
and    $T  \subseteq \ZZ_n$  
such that the symbol of $\G$ is of the form
$[S,\widehat{S},T]$, and moreover  
$|S|= m(m+1) = (n-1)/2$ and $|T|=m^2 =(n - \sqrt{2n-1})/2 $.
\end{theorem}

Of course, the Petersen graph and its complement are examples of strongly regular bicirculants (arising from the action of $S_5$ on $10$ pairs 
of a $5$-element set). It is a direct  consequence of CFSG
that no other rank $3$ strongly regular $p$-bicirculant, $p$ a prime,  exists.
There are known constructions of strongly regular 
$n$-bicirculants, $n$ odd, for $n \in  \{13, 25, 41, 61\}$ (two for $n=41$) which do not arise from groups of rank $3$ (see \cite{MMS07}). 
With the exception of $n=13$ and  one of the constructions for $n=41$, 
all other graphs are vertex-transitive.

As expected from Corollary~\ref{cor:3iso-2} of this paper
none of these graphs is $3$-isoregular.  
In the example below this is shown for the Petersen graph.

\begin{example}
\label{ex:pet}
{\rm 
Note that as a rank $3$-graph the Petersen graph 
automatically satisfies the $t$-vertex condition for every $t \leq 10$.
But it is not $3$-isoregular; the argument is fairly straightforward. 
Namely since its girth is $5$ it is clear that  any subset of three vertices with at least one edge (in the induced subgraph) has no neighbors. 
On the other hand, independent subsets with three vertices are of two different types. Adopting one of the usual notations consistent with the above mentioned action of $S_5$ on pairs from 
$\{1,2,3,4,5\}$ with adjacency following the empty intersection rule, we have that the two subsets $\{12,13,23\}$ and $\{12,13,14\}$ have, 
one  neighbor and no neighbors, respectively. 
}
\end{example}

As mentioned in the introductory section there are two $3$-isoregular $8$-bicirculants. 
There is an additional strongly regular $8$-bicirculant, known
as the Shrikhande graph, but it is not $3$-isoregular. 
Their respective symbols are given in the example below.

\begin{example}
\label{ex:8-bicirculants}
{\rm 
The (bicirculant) symbol of the Clebsch graph 
 is 
$[\{ \pm 1, 4 \}, \{ \pm 3, 4 \}, \{ 0,2 \}]$.
To see that it is $3$-isoregular note that $\lambda=0$
implies that it is a triangle-free graph
and so, with the exception of three independent vertices,
any other triple of vertices has valency zero. 
It can be checked that the valency of three independent vertices is one.

Next, the symbol of $K_4 \square K_4$ is
$[\{ \pm 1 \}, \{ \pm 3 \}, \{ 0,1,3,4\}]$.
To see that it is $3$-isoregular we use the fact that
it is vertex-transitive, arc-transitive and distance-transitive,
and compute the valencies of the subgraphs isomorphic to
$K_3$, $K_{1,2}$, $K_2+K_1$ and $\bar{K}_3$
to be respectively $1$, $0$, $1$ and $0$.

Finaly, the Shrikhande graph has two different representations as an
$8$-bicirculant with respective symbols being  
$[\{ \pm 1 \}, \{ \pm 3 \}, \{ 0,\pm 1,4 \}]$
and
$[\{ \pm 1, \pm 2\}, \{ \pm 2, \pm 3 \}, \{ 1,3 \}]$.
It is not $3$-isoregular as it admits induced
subgraphs $K_{1,2}$ of different valencies, namely $0$ and $1$.
}
\end{example}

\subsection{Tricirculants}
\label{sub:tri}
\indent

Strong regularity of tricirculants has also been a matter of research interest.
In \cite{KMMS09} certain structural results for  such graphs were obtained. In short, based on arithmetic conditions on the valency of subgraphs induced on the three orbits of a $(3,n)$-semiregular automorphism  two types of such graphs were identified (up to taking the complement).  
The first type is characterized by the fact that the induced subgraphs
on the three orbits are of the same valency, while in the second type
this is not the case.

\begin{proposition}
\label{pro:tricirculants}
{\rm \cite[Theorem 5.3]{KMMS09}}
Let $\G$ denote a strongly regular $n$-tricirculant with parameters $(3n,k,\lambda,\mu)$ where  either $n$ is a prime number or $n$ is coprime to $6\sqrt{(\lambda-\mu)^2+4(k-\mu)}$. 
Assume further that there is at least one edge in $\G$ with both endvertices
in the same orbit of the $(3,n)$-semiregular automorphism $\sigma$ of $\G$ 
and that the number of edges of $\G$ 
with both endvertices  in the same orbit of  
$\sigma$ in $\G$ is smaller than or equal to the number of edges 
with both endvertices  in the same orbit of $\sigma$ in 
$\overline{\G}$. Then there exists an integer $s$, such that either
$$
 (3n,k,\lambda,\mu)=(3(12 s^2 + 9s + 2), (4s + 1)(3s + 1), s(4s + 3), s(4s + 1)),
$$
or
$$
  (3n,k,\lambda,\mu)=(3(3 s^2 - 3s + 1), s(3s - 1), s^2 + s - 1, s^2).
$$
\end{proposition}

In \cite{KMMS09}  constructions of graphs 
were given for $n \in  \{5,7,19\}$.
None of these graphs is $3$-isoregular, as checked by Magma \cite{Mag}.   
A brief explanation is given below for the case $n=5$. 

\begin{example}
\label{ex:15}
{\rm 
The graph on $15$ vertices is
the point graph of the generalized quadrangle $GQ(2,2)$.
Its  valency is $6$. Alternatively, it can  be thought of as the complement of the line graph $L(K_6)$ of the complete graph on $6$ vertices, or, as a third possibility,  
as the line graph of the Petersen graph with added triangles inside each of the $5$  blocks of imprimitivity (of size $3$).  

With regards to the number of neighbors,
there are two kinds of independent sets with three vertices,
respectively, with one  neighbor and no neighbors.
Using the ``line  graph of the Petersen graph'' representation,
this graph  is a $3$-fold cover of $K_5$ with triangles inside the fibres,
which are blocks of imprimitivity of the group $SL(2,4)$ acting on the projective line $PG(1,4) = \{\infty, 0,1,w,w^2\}$ over $GF(4)$,
where $1+w =w^2$. 
The vertex set is identified with
$$
\{(x,i) \mid x \in PG(1,4), i \in \ZZ_3\},
$$
and the voltages satisfy the following conditions:
voltage of an edge in $K_5$ with one endvertex
$\infty$ is $id \in D_{6}$, while the other voltages 
are reflections $\tau_i = (i) (i-1,i+1)$,  
$i \in \ZZ_3$. More precisely, the voltage between
blocks 
$0$ and $w^i$ is  $\tau_i$ and between blocks $w^i$ and $w^j$ is 
$\tau_k$  if and only if $w^k = w^i + w^j$.
The representatives of the two types of 
independent sets of $3$ vertices are $\{ (\infty,0), (0,1), (1,1)\}$ and
$\{ (\infty,0), (0,1), (w,2)\}$. While the latter has no neighbor, 
$(\infty,1)$ is the neighbor of the first set.
}
\end{example}


\section{Locally $3$-isoregular graphs}
\label{sec:locally}
\indent

In this section we introduce the concept of locally $3$-isoregular graphs. 
Let $\G$ denote a finite, simple, connected graph. We say that 
a pair of distinct vertices $(x,y)$ of $\G$ is  
{\em $3$-isoregular} if and only if the valency of any set $T = \{x,y,z\}$, $z \neq x,y$,
depends solely on the isomorphism type of the subgraph induced on $T$.
Furthermore, we say that $\G$ is {\em locally $3$-isoregular at vertex} $x$, if 
there exist  an edge $(x,y)$ and a non-edge $(x,z)$, both $3$-isoregular.
And moreover, we say that $\G$ is {\em locally $3$-isoregular}
if there exists a vertex $x$ of $\G$ such that $G$ is locally
$3$-isoregular at $x$.
Clearly, in a  3-isoregular graph all edges and non-edges are $3$-isoregular.
Note that if $\G$ has a 3-isoregular edge $(x,y)$, then its complement $\overline{\G}$ has a 3-isoregular non-edge $(x,y)$. It follows that $\G$ is locally $3$-isoregular at vertex $x$ if and only if $\overline{\G}$ is locally $3$-isoregular at $x$.

Assume now that the edge $(x,y)$ of $\G$ is $3$-isoregular. 
Then there exist non-negative integers 
$Q$, $R$, $W$ which are the valencies  
of a set $\{x,y,z\}$, $z\in V(\G)$, depending, respectively, on whether the subgraph induced on $\{x,y,z\}$ is isomorphic to the triangle $K_3$, the 
$2$-claw $K_{1,2}$, or the graph $K_2 + K_1$. We say that $Q$, $R$, $W$ are {\em parameters associated with} the edge $(x,y)$. The next result
gives two relations on the parameters $Q$, $R$ and $W$.

\begin{proposition}
	\label{pro:3iso-eq1}
Let $\G$ be a strongly regular graph with parameters $(n,k,\lambda,\mu)$
and let $(x,y)$ be a $3$-isoregular edge of $\G$ with $Q, R, W$ as the associated parameters. Then the following (i)-(ii) hold:
	\begin{itemize}
\itemsep=0pt
		\item[(i)]
		 $
		   \lambda(\lambda-Q-1)=R(k-\lambda-1),
		 $
		 \item[(ii)]
		 $
		 \lambda\mu(k-2\lambda+Q)=W(k-\mu)(k-\lambda-1).
		 $
	\end{itemize}
Consequently, $W(k-\mu)=\mu(\lambda-R)$.
\end{proposition}

\begin{proof}
 For a vertex $z$ of $\G$ and for a non-negative integer $i$, we set $\G_i(z)$ to denote the set of vertices of $\G$ at distance $i$ from $z$. Furthermore, for $0 \le i,j \le 2$, we set 
 $$
   D_j^i = D_j^i(x,y) = \G_i(x) \cap \G_j(y).
 $$ 
 Note that $D^0_1=\{x\}$, $D^1_0=\{y\}$, $D_0^0=\emptyset$, and that we have
$$
  |D_1^1|=\lambda, \quad |D^1_2|=|D^2_1|=k-\lambda-1, \quad |D^2_2|=\frac{(k-\mu)(k-\lambda-1)}{\mu}.
 $$

 In order to prove (i), observe that by the definition of parameter $R$, every vertex $z \in D^2_1$ has exactly $R$ neighbors in $D_1^1$, and consequently exactly $\lambda-Q-1$ neighnours in $D^2_1$. Now if $D^1_1=\emptyset$, then $R=\lambda=0$, and the above equality holds. Assume therefore that $D^1_1 \ne \emptyset$. Then, by the definition of parameter 
$Q$, every vertex $z \in D^1_1$ has exactly $Q$ neighbors in $D_1^1$.  Now counting the number of edges between $D^1_1$ and $D^2_1$ in two different ways (taking into account the above formulae for the cardinalities of these sets), we get the desired equality.
 
As for (ii), observe that by the definition of parameter $W$, every vertex $z \in D^2_2$ has exactly $W$ neighbors in $D_1^1$. Similarly as above, counting the number of edges between $D^1_1$ and $D^2_2$ in two different ways (taking into account the above formulae for the cardinalities of these sets), we get the desired equality.

The last claim of the Proposition no follows immediately from (i), (ii) above.
\end{proof}

Assume now that a non-edge $(x,z) \; (x \ne z)$ of $\G$ is $3$-isoregular.
Then there exist non-negative integers 
$R'$, $W'$, $V$  which are the valencies  
of a set $\{x,z,u\}$, $u\in V(\G)$, depending, respectively, on whether the subgraph induced on $\{x,z,u\}$ is isomorphic to the $2$-claw $K_{1,2}$, the graph $K_2 + K_1$, or the graph $3K_1$. 
In this case we say that $R',W',V$ are {\em parameters associated with}  the non-edge $(x,z)$. The next result
gives two relations on the parameters $R', W'$ and $V$.

\begin{proposition}
	\label{pro:3iso-eq2}
Let $\G$ be a strongly regular graph with parameters $(n,k,\lambda,\mu)$ 
and let $(x,z)$ be a $3$-isoregular non-edge of $\G$ with $R', W', V$ as
the associated parameters. Then the following (i)-(ii) hold:
	\begin{itemize}
\itemsep=0pt
		\item[(i)]
		$
		\mu (\lambda - R')=(k-\mu)W',
		$
		\item[(ii)]
		$
		\mu(k-2-2\lambda+R') = V\Big(\frac{k(k-\lambda-1)}{\mu}-k+\mu-1 \Big).
		$
	\end{itemize}
\end{proposition}

\begin{proof}
The argument is analogous to the one used in the 
proof of Proposition \ref{pro:3iso-eq1}. 
\end{proof}

\begin{proposition}
	\label{pro:3iso-eq3}
Let $\G$ be a strongly regular graph with parameters $(n,k,\lambda,\mu)$. Further let $\G$ be locally 3-isoregular at vertex $x$. Let $y$ be a neighbor of $x$ such that $(x,y)$ is $3$-isoregular, and let $Q,R,W$ be the associated parameters. Let $z$ be a non-neighbor of $x$ such that $(x,z)$ is $3$-isoregular, and let $R',W',V$ be the associated parameters. Then $R=R'$ and $W=W'$.
\end{proposition}

\begin{proof}
	For $0 \le i,j \le 2$ we set 
	$$
	  D^i_j = D^i_j(x,z) = \G_i(x) \cap \G_j(z).
	$$
	Note that we have
	$$
	  |D^1_1| = \mu, \quad |D^1_2| = |D^2_1| = k-\mu, \quad |D^2_2|=\frac{k(k-\lambda-1)}{\mu}-k+\mu-1.
	$$
	Observe that as $(x,z)$ is 3-isoregular, every vertex $v \in D^1_1$ has exactly $R'$ neighbors in $D^1_1$, every vertex $v \in D^1_2 \cup D^2_1$ has exactly $W'$ neighbors in $D^1_1$, and every vertex $v \in D^2_2$ has exactly $V$ neighbors in $D^1_1$. We split the proof into two cases, depending on weather $y$ is a neighbor of $z$ (that is, $y \in D^1_1$), or a non-neighbor of $z$ (that is, $y \in D^1_2$). 
	
	Assume first that $y \in D^1_1$. By the comments above, $y$ must have $R'$ neighbors in $D^1_1$. However, since $(x,y)$ is 3-isoregular, $x,y,z$ have $R$ common neighbors, implying that $R=R'$. By Proposition \ref{pro:3iso-eq2}(i) we therefore have
	$$
	   \mu(\lambda-R)=(k-\mu)W'.
	$$
	The equality $W=W'$ now follows from Proposition \ref{pro:3iso-eq1}.
	
If $y \in D^1_2$ then we have that $W=W'$ 
using an argument analogous to the one that was used in the previous
case to obtain  $R=R'$. By Proposition \ref{pro:3iso-eq2}(i) we therefore have
	$$
	\mu(\lambda-R')=(k-\mu)W,
	$$
and again using Proposition~\ref{pro:3iso-eq1} we get $R=R'$.
\end{proof}


\section{Non-existence of local $3$-isoregular strongly regular $n$-bicirculants, $n$ odd}
\label{sec:bicirc}
\indent

Let $\G$ denote a strongly regular $n$-bicirculant of order $2n$, $n$ odd. Recall that, by Theorem~\ref{the:bicirc}, parameters of $\G$ satisfies 
$$
  (2n,k,\lambda,\mu) = (2(2m^2+2m+1),m(2m+1),m^2-1,m^2),
$$
or 
$$
  (2n,k,\lambda,\mu) = (2(2m^2+2m+1)(m+1)(2m+1),m(m+2),(m+1)^2)
$$
for some positive integer $m$ 
(see also \cite{deRJ92,DM88,HW56,HW64}). In this section we show that such a graph is not locally $3$-isoregular at any vertex $x$.  Note that if $\G$ has parameters $(2(2m^2+2m+1),m(2m+1),m^2-1,m^2)$, then $\overline{\G}$ has parameters $(2(2m^2+2m+1),(m+1)(2m+1),m(m+2),(m+1)^2)$.

If $m=1$, then the above two parameter sets correspond to the Petersen graph and its complement. Note that every edge of the Petersen graph is 3-isoregular (and so every non-edge of its complement is also 3-isoregular). On the other hand, none of the non-edges of the Peterson graph is $3$-isoregular (and so none of the edges of its complement is 3-isoregular). Therefore, for the rest of this section we may assume $m \ge 2$. 

\begin{proposition}
	\label{prop:bicirc1}
	Let $\G$ be a strongly regular bicirculant with parameters $(2(2m^2+2m+1),m(2m+1),m^2-1,m^2)$, where $m \ge 2$, and let $(x,y)$ be
a 3-isoregular edge of $\G$ with
the associated parameters $Q,R,W$. Then $m$ is odd and 
		$$
		  Q= \frac{m^2-m-4}{2}, \quad R=\frac{m^2-1}{2} , \quad W=\frac{m(m-1)}{2}.
		$$
\end{proposition}

\begin{proof}
	Using Proposition \ref{pro:3iso-eq1}(i) we find that 
	$$
	  R=\frac{(m-1)(m^2-Q-2)}{m},
	$$
	and so $Q+2=\alpha m$ for some positive integer $\alpha$ (recall that $Q+2 \ge 2$). It follows that $R=(m-1)(m-\alpha)$. Since $R \ge 0$, we have that $1 \le \alpha \le m$. 	
	
We claim that $\alpha < m$. Assume to the contrary that $\alpha=m$. Then $Q=m^2-2$ and $R=0$. Let $D^1_1=\G_1(x) \cap \G_1(y)$ and note that $|D^1_1| = \lambda = m^2-1$. As $Q=m^2-2$, this shows that $D^1_1 \cup \{x,y\}$ is a clique of cardinality $m^2+1$ in $\G$. But this contradicts the Hoffman bound \eqref{hoffman}. This shows that $\alpha < m$.
	
	Using Proposition \ref{pro:3iso-eq1}(ii) we find that 
	$$
	W=\frac{(\alpha+1)(m-1)m}{m+1}.
	$$
If $m$ is even then $\gcd(m-1,m+1)=1$. It follows that $m+1$ divides $\alpha+1$, which is impossible as $1 \le \alpha \le m-1$.  If $m$ is odd then
$\gcd(m-1,m+1)=2$. It follows that $(m+1)/2$ divides $\alpha+1$, and so 
$\alpha=(m-1)/2$, which gives us that $R=(m^2-1)/2$. 
The expressions for $Q$ and $W$ then follow using Proposition~\ref{pro:3iso-eq1}.
\end{proof}

\begin{proposition}
	\label{prop:bicirc2}
	Let $\G$ be a strongly regular bicirculant with parameters $(2(2m^2+2m+1),m(2m+1),m^2-1,m^2)$ with $m \ge 1$, and let $x$ be a vertex of $\G$. Then $\G$ is not locally $3$-isoregular at $x$.
\end{proposition}

\begin{proof}
	If $m=1$, then $\G$ is the Petersen graph, and 
the result follows in view of the discussion in the beginning of this section.
Assume therefore that $m \ge 2$ and that $\G$ is locally $3$-isoregular at $x$. Let $y,z\in V(\G)$  such that $(x,y)$ is a $3$-isoregular edge, and $(x,z)$ is a $3$-isoregular non-edge of $\G$. 
Then by Proposition~\ref{pro:3iso-eq3} there exist nonnegative integers $Q,R,W,V$ 
such that the triples
 $Q,R,W$ and $R,W,V$ are the respective associated parameters
for $(x,y)$ and $(x,z)$. Then, by Proposition \ref{prop:bicirc1},
it follows that $m$ is odd and $Q=(m^2-m-4)/2$, $R=(m^2-1)/2$, $W=(m(m-1))/2$. Finally, using Proposition \ref{pro:3iso-eq2}(ii) 
we now find that 
	$$
	  V = \frac{m(m^2+2m-1)}{2(m+2)}.
	$$ 
Since $m$ is odd, this implies that $m+2$ divides $m^2+2m-1=m(m+2)-1$, a contradiction.
\end{proof}

\begin{theorem}
	\label{thm:bicirc3}
	Let $\G$ be a strongly regular $n$-bicirculant, $n$ odd,
and let $x\in V(\G)$.
Then  $\G$ is not locally $3$-isoregular  at $x$.
\end{theorem}

\begin{proof}
Recall that $\G$ is locally $3$-isoregular at $x$ if and only if $\overline{\G}$ is locally $3$-isoregular at $x$. Since either $\G$ or $\overline{\G}$ has the strong regularity parameters $(2(2m^2+2m+1),m(2m+1),m^2-1,m^2)$,  
$m \ge 1$, the result follows from Proposition \ref{prop:bicirc2}.
\end{proof}

The result about non-existence of $3$-isoregular $n$-bicirculants
in the $n$ odd case, is now immediate.

\begin{corollary}
	\label{cor:3iso-2}
	There exists no $3$-isoregular $n$-bicirculant for $n$ odd.
\end{corollary}


\section{On non-existence of $3$-isoregular $n$-bicirculants for $n$ even}
\label{sec:addendum}
\indent

The aim of this section is to lay out the ground for future research on $3$-isoregularity and more generally strong regularity of $n$-bicirculants, $n$ even.  
The following result of Leung and Ma \cite{LM93} is essential in this respect.
Note that the `partial difference triple'  refers to the triple $(C,D,D')$ where $C=T$, $D=S$ and $D'=S'$
in our notation for bicirculants 
in Subsection~\ref{sub:bi}. 
Also, $c=|C|=|T|$, $d=|D|=|D'|=|S|=|S'|$.

\begin{proposition}\label{pro:LM93}
{\rm \cite[Theorem~3.1]{LM93}}
Up to complementation, the parameters for any non-trivial partial
difference triples in cyclic groups are the following:
\begin{enumerate}[(a)]
\itemsep=0pt
\item $(n;c,d;\lambda,\mu)=(2m^2+2m+1; m^2,m^2+m;m^2-1,m^2)$ where $m\ge 1$.
\item $(n;c,d;\lambda,\mu)=(2m^2; m^2,m^2-m;m^2-m,m^2-m)$ where $m\ge 2$.
\item $(n;c,d;\lambda,\mu)=(2m^2; m^2,m^2+m;m^2+m,m^2+m)$ where $m\ge 3$.
\item $(n;c,d;\lambda,\mu)=(2m^2; m^2\pm m,m^2;m^2\pm m,m^2\pm m)$ where $m\ge 2$.
\end{enumerate}
\end{proposition}

By Proposition~\ref{pro:LM93} there are four families of strongly regular $n$-bicirculants for $n$ even (essentially two when one only takes into account parameters $k$, $\lambda$ and $\mu$ and  disregards the cardinalities of sets $C$, $D$ and $D'$ in their notation, 
respectively, $T$, $S$ and $S'$ in our notation). Namely,
both (b) and (c)  in Proposition~\ref{pro:LM93} have their counterparts in the two variations of (d). So, since $k=c+d$, the parameters for such a graph $\Gamma$ are as follows, respectively for (b) and (c):
$$
(2n,k,\lambda,\mu) = (4m^2, 2m^2-m, m^2-m, m^2-m)
$$
\noindent
and
$$
(2n,k,\lambda,\mu) = (4m^2, 2m^2+m, m^2+m, m^2+m)
$$

The goal of this section is to  prove that no locally $3$-isoregular
$n$-bicirculant, $n =2m^2$, exists for $m$ odd.
Using Propositions \ref{pro:3iso-eq1} and {\ref{pro:3iso-eq2} we get the following expressions for parameters $Q$, $W$ and $V$ in terms of $R$, $k$ and $\lambda$ (using that $\lambda = \mu$).

$$
\lambda Q = \lambda (\lambda -1)  -(k-\lambda -1) R.
$$

$$
(k-\lambda) W = \lambda (\lambda - R).
$$

$$
V\Big(\frac{k(k-\lambda-1)}{\lambda}-k+\mu-1 \Big) =
\lambda ( k-2-2\lambda + R).
$$

We deal first with case (b).

\begin{proposition}
\label {pro:(b)}
For $m$ odd there are no locally $3$-isoregular graphs  in family (b) of 
Proposition~\ref{pro:LM93}.
\end{proposition}

\begin{proof}
We have that
$k = 2m^2-m,$ and $\lambda = m^2-m$, giving us the following 
expressions for the above parameters:

\begin{equation}
\label{equ:Q}
Q = m^2 -m-1 - \frac {(m+1)R}{m},
\end{equation}

\begin{equation}
\label{equ:W}
W = (m-1)^2 - \frac {(m-1)R}{m},
\end{equation}

\begin{equation}
\label{equ:V}
 V =  \frac{m(m-2+R)}{m+2}.
\end{equation}

\noindent
Now since all of these parameters are non-negative integers we get that $m$ divides $R$, and  also that  $R \leq m^2-m$ and that 
$R \leq \frac{m(m^2-m-1)}{m+1}=(m-1)^2-\frac{1}{m+1} < (m-1)^2$. The second inequality gives us that  
$R \leq m(m-2)$. Hence we can say that there is a non-negative integer $\alpha$ such that  $R = \alpha m$ and furthermore 

\begin{equation}
\label{equ:alpha}
\alpha \leq m-2.
\end{equation}

\noindent
Apply this for the expression of $V$ in  (\ref{equ:V}) to get that

$$
V =\frac{m^2-2m +mR}{m+2} = \frac{m^2+ 2m -4m +mR}{m+2} =
m + \frac{m(R-4)}{m+2} = m + \frac{m(\alpha m - 4)}{m+2}.
$$

\noindent
Since we assume that $m$ is odd we have that $m$ and $m+2$ are coprime and so  $m+2$ is a divisor of 
$\alpha m - 4 = (\alpha m + 2 \alpha - 2\alpha  - 4)$ 
and so $m+2$ is a divisor of $2\alpha + 4$ and so also of $\alpha +2$. 
In particular we have that
$m \leq \alpha$ contradicting (\ref{equ:alpha}).
\end{proof}

We deal next with case (c).

\begin{proposition}
\label {pro:(c)}
For $m$ odd there are no locally $3$-isoregular graphs in family (c) of 
Proposition~\ref{pro:LM93}.
\end{proposition}

\begin{proof}
Assume to the contrary that there exists $\Gamma$ in family (c) of Proposition~\ref{pro:LM93}, that is locally $3$-isoregular at vertex $x$. Fix a neighbour $y$ and a non-neighbour $z$ of $x$, such that $(x,y)$ and $(x,z)$ are $3$-isoregular. Let $Q,R,W,V$ be the associated parameters. Recall  that $k = 2m^2+m,$ and $\lambda = m^2+m$, giving us the following expressions for the above parameters:

\begin{equation}
\label{equ:Qc}
Q = m^2 + m-1 - \frac {(m-1)R}{m},
\end{equation}

\begin{equation}
\label{equ:Wc}
W = (m+1)^2 - \frac {(m+1)R}{m},
\end{equation}

\begin{equation}
\label{equ:Vc}
 V =  \frac{m(R - m-2)}{m-2}.
\end{equation}

\noindent
Again there exists a non-negative integer $\alpha$
such that $R = \alpha m$. 
Using this fact in the expression for $V$  in (\ref{equ:Vc})
we have, with an analogous argument as in the previous proposition,
that $m-2$ divides $2\alpha -4$. Since $m$ is odd this implies that
$m-2$  divides $\alpha -2$, and so $m \le \alpha$.
Furhermore, note that it follows from the definition of parameter $V$ that $V \le \mu$. Therefore, \eqref{equ:Vc} implies that $R = \alpha m \le m^2$, and consequently $\alpha \le m$. The above two inequalities thus imply that $\alpha=m$, and so we have the following values for $Q$, $R$, $W$ and $V$:
$$
R = m^2, \quad Q = 2m-1, \quad W = m+1, \quad V = m(m+1).
$$

\noindent
We will now argue that this is not possible. First note that the eigenvalues of  $\Gamma$ are $k=2m^2+m$ and 
$$
  r,s = \frac{(\lambda-\mu) \pm \sqrt{(\lambda-\mu)^2+4(k-\mu)}}{2}= \pm \sqrt{k-\mu} = \pm m.
$$
By the Hoffman's bound, any clique $C$ in $\Gamma$ satisfies 
$$
  |C| \le 1+\frac{k}{m} = 2+2m.
$$
Recall the sets $D_1^1=\Gamma(x) \cap \Gamma(z)$ and $D^1_2=\Gamma(x) \cap \Gamma_2(z)$, and observe that $\Gamma(x)=D^1_1 \cup D^1_2$.  Note that $|D^1_1|=\mu=m^2+m$, and consequently $|D^1_2|=k-|D^1_1| = 2m^2+m-m^2-m=m^2$. Pick a vertex $w \in D^1_2$ and note that $w$ has exactly $W=m+1$ neighbours in $D^1_1$. Since $x$ and $w$ must have exactly $\lambda=m^2+m$ common neighbours, it follows that $w$ has exactly $m^2-1$ neighbours in $D^1_2$. As $|D^1_2|=m^2$, this implies that the subgraph of $\Gamma$, induced on $\{x\} \cup D^1_2$, is a clique with cardinality $1+m^2$. Using the above Hoffman's bound we now get that
$$
  1 + m^2 \le 2+2m,
$$
a contradiction as $m \ge 3$ by Proposition \ref{pro:LM93}.

\end{proof}

\begin{remark}
{\rm 
Of course, the same approach follows also in the case $m$ even but here we do have to consider also the possibility that $(m\pm 2)/2$, respectively, 
divide $\alpha \pm 2$ and so  the computation for the parameters $Q$, $R$, $W$ and $V$ gives  us the following, respectively for cases (b) and (c):

$$
Q = W = (m^2-m)/2,  R =V = (m^2 -2m)/2
$$

\noindent
and

$$
Q = W = (m^2+m)/2,  R =V = (m^2 +2m)/2.
$$
}
\end{remark}

\begin{remark}
{\rm In  view of Propositions~\ref{pro:(b)} and~\ref{pro:(c)} 
a $3$-isoregular $n$-bicirculant, $n$ even, 
may only exist  when $n=2m^2$ for $m$ even.
As mentioned in Example~\ref{ex:8-bicirculants} 
the Clebsch graph and $K_4 \square K_4$ are
$3$-isoregular $8$-bicirculants, fitting the above requirement for $m=2$.
The next possibility for such a $3$-isoregular bicirculant
could occur for $m=4$. Note that in this case the parameters of a potential $3$-isoregular bicirculant would be $(64,28,12,12)$ or $(64,36,20,20)$. While the exact number of strongly regular graphs with these parameters is not known, the results in \cite{FI} show that there is at least 11,063,360 strongly regular graphs with parameters $(64,28,12,12)$, and at least 8,613,977 strongly regular graphs with parameters $(64,36,20,20)$.}
\end{remark}


\section{Local $3$-isoregularity of strongly regular  tricirculants}
\label{sec:tri}
\indent

Let $\G$ denote a strongly regular $n$-tricirculant with parameters $(3n,k,\lambda,\mu)$ where  either $n$ is a prime number or $n$ is coprime to $6\sqrt{(\lambda-\mu)^2+4(k-\mu)}$. 
Assume further that there is at least one edge in $\G$ with both endvertices
in the same orbit of the $(3,n)$-semiregular automorphism $\sigma$ of $\G$ 
and that the number of edges of $\G$ 
with both endvertices  in the same orbit of  
$\sigma$ in $\G$ is smaller than or equal to the number of edges 
with both endvertices  in the same orbit of $\sigma$ in 
$\overline{\G}$. 
In this section we show that $\G$ is not locally $3$-isoregular with respect to any vertex $x$ of $\G$.
First recall that, by Proposition~\ref{pro:tricirculants}, 
there exists an integer $s$, such that either
$$
 (3n,k,\lambda,\mu)=(3(12 s^2 + 9s + 2), (4s + 1)(3s + 1), s(4s + 3), s(4s + 1)),
$$
or
$$
  (3n,k,\lambda,\mu)=(3(3 s^2 - 3s + 1), s(3s - 1), s^2 + s - 1, s^2).
$$

\begin{proposition}
	\label{prop:tri1}
With notation as above, assume that 
the parameters of $\G$ satisfy 
	$$
	(3n,k,\lambda,\mu)=(3(12 s^2 + 9s + 2), (4s + 1)(3s + 1), s(4s + 3), s(4s + 1)).
	$$
Then an edge $(x,y)$ of $\G$ is $3$-isoregular if and only if $s=-1$, in which case $\G$ is isomorphic to the complement of the triangular graph $T(6)$. 
\end{proposition}

\begin{proof}
Assume that the edge $(x,y)$ is $3$-isoregular and let $Q,R,W$ denote the associated parameters. Observe that if $s =0$, then $3n=6$, $k=1$, and so $\G$ is isomorphic to $3K_2$, contradicting our assumptions that $\G$ is connected. Therefore, $s \ne 0$. 
	Using Proposition \ref{pro:3iso-eq1}(i) we get
	$$
	R=\frac{(4s+3)(4 s^2+3 s-Q-1)}{4(2s+1)}.
	$$
	Note that $(4s+3)-2(2s+1)=1$, and so for every integer $s$ we have $\gcd(4s+3,2s+1)=1$. Now clearly $4s+3$ is odd, and so also $\gcd(4s+3,4(2s+1))=1$. It follows that $4(2s+1)$ divides $4s^2+3s-Q-1$, that is,
	$$
	Q = 4s^2 + 3s - 1 - 4 \alpha (2s+1)
	$$
	for some integer $\alpha$. Note that this yields $R=\alpha(4s+3)$. Using Proposition \ref{pro:3iso-eq1}(ii) we get
	$$
	W=\frac{s(s-\alpha)(4s+3)}{2s+1}.
	$$
    Observe that $\gcd(s(4s+3),2s+1)=1$, and so $2s+1$ divides $s-\alpha$. It follows that $\alpha=s-\beta(2s+1)$ for some integer $\beta$. This gives us 
    $$
      W = \beta s (4s+3), \quad R=-(4s+3)(2\beta s - s + \beta),
    $$
    and so $W \ge 0$ implies $\beta \ge 0$.  Now if $s \ge 1$, then $R \ge 0$ gives us $\beta \le s/(2s+1)< 1$. Similarly, if $s \le -1$, then again  $R \ge 0$ gives us $\beta \le s/(2s+1)\le 1$. This shows that either $\beta=0$, or 
$s=-1$ and $\beta=1$. If $\beta=0$ then $Q=-(4s^2+s+1)$, contradicting $Q \ge 0$. If $s=-1$ and $\beta=1$, then $\G$ is the complement of the triangular graph $T(6)$, see \cite[Section 6]{KMMS09}. It is easy to check 
(using the fact that $\G$ is distance-transitive) that $(x,y)$ is indeed $3$-isoregular  with $Q=0, R=0, W=1$.
\end{proof}

\begin{proposition}
	\label{prop:tri2}
With notation as above, assume that 
the parameters of $\G$ satisfy 
	$$
	(3n,k,\lambda,\mu)=(3(3 s^2 - 3s + 1), s(3s - 1), s^2 + s - 1, s^2).
	$$
Then none of the edges of $\G$ is $3$-isoregular.
\end{proposition}

\begin{proof}
Suppose to the contrary that an edge $(x,y)$ is $3$-isoregular and let $Q,R,W$ denote the associated  parameters. Observe that if $s =1$, then $3n=3$, and so $\G$ is isomorphic to $K_3$, contradicting our assumptions 
that there exists an edge with both endvertices in the same orbit of
$\sigma$.
 If $s \in \{0,-1\}$, then $\lambda=-1$, a contradiction. Therefore, $s \not \in \{-1,0,1\}$. 
	Using Proposition \ref{pro:3iso-eq1}(i) we get
	$$
	  R=\frac{(s^2+s-1)(s^2+s-Q-2)}{2(s-1)s}.
	$$
	Note that $(2s-1)(s^2+s-1)-(2s+3)s(s-1)=1$, and so for every integer $s$ we have $\gcd(s^2+s-1,s(s-1))=1$. Now clearly $s^2+s-1$ is odd, and so also $\gcd(s^2+s-1,2s(s-1))=1$. It follows that $2s(s-1)$ divides $s^2+s-Q-2$, that is,
	$$
	  Q = s^2 + s - 2 - 2 \alpha s (s-1)
	$$
	for some integer $\alpha$. Note that this yields $R=\alpha(s^2+s-1)$. Since $R$ is non-negative, we get that $\alpha \ge 0$. Using Proposition \ref{pro:3iso-eq1}(ii) we get
	$$
	W=-\frac{(\alpha-1)s(s^2+s-1)}{2s-1}.
	$$
	Observe that $s(s^2+s-1)/(2s-1)$ is positive, and since $W$ is non-negative this yields $\alpha \le 1$. It follows that $\alpha \in \{0,1\}$. 
	
	If $\alpha=1$, then $0 \le Q=-(s-1)(s-2)$ implies $s=2$. In this case $\G$ is isomorphic to the triangular graph $T(7)$, see \cite[Section 6]{KMMS09}.  It is easy to see, however, that in $T(7)$ none of its edges is $3$-isoregular (note that $T(7)$ is distance-transitive).
	
	If on the other hand $\alpha=0$, then we have $W=s(s^2+s-1)/(2s-1)$. It follows that $2s-1$ divides $s(s^2+s-1)$. But now
	$$
	  \frac{8s(s^2+s-1)}{2s-1} = 4s^2+6s-1-\frac{1}{2s-1}
	$$
	implies that $2s-1$ divides $1$, a contradiction.
\end{proof}

\begin{theorem}
	\label{thm:tri3}
Let $\G$ be as a strongly regular $n$-tricirculant 
with parameters $(3n,k,\lambda,\mu)$ 
where  either $n$ is a prime number or $n$ is 
coprime to $6\sqrt{(\lambda-\mu)^2+4(k-\mu)}$. 
Then for every vertex $x\in V(\G)$, it follows that
$\G$ is not locally $3$-isoregular at $x$.
\end{theorem}

\begin{proof}
	Assume first that $\G$ is the complement of the triangular graph $T(6)$. Then it is easy to see that none of the non-edges of $\G$ is $3$-isoregular, implying that neither of $\G$ and $\overline{\G}$ is locally $3$-isoregular at $x$. Assume now that $\G$ is not isomorphic to the triangular graph $T(6)$ or its complement. If the number of edges of $\G$ that are contained in the orbits of a $(3,n)$-semiregular automorphism $\sigma$ of $\G$ is smaller than or equal to the number of edges of $\overline{\G}$ contained in the orbits of $\sigma$, then Propositions \ref{prop:tri1} and \ref{prop:tri2} shows that $\G$ is not $3$-isoregular at $x$.  If the number of edges of $\G$ that are contained in the orbits of $\sigma$ is greater than the number of edges of $\overline{\G}$ contained in the orbits of $\sigma$,  then Propositions \ref{prop:tri1} and \ref{prop:tri2} shows that none of the non-edges $(x,y)$ of $\G$ is $3$-isoregular, and so again $\G$ is not $3$-isoregular at $x$.
\end{proof}

\begin{corollary}
	\label{cor:3iso-3}
Let $\G$ be a strongly regular $n$-tricirculant 
with parameters $(3n,k,\lambda,\mu)$ 
where  either $n$ is a prime number or $n$ is 
coprime to $6\sqrt{(\lambda-\mu)^2+4(k-\mu)}$. 
Then  
$\G$ is not $3$-isoregular. 
\end{corollary}


\end{document}